\documentclass[a4paper]{amsart}
\usepackage{lmodern}
\usepackage{amsthm}
\usepackage{amssymb}
\usepackage{amsmath}
\usepackage{hyperref}
\hypersetup{
  unicode=true,
  colorlinks=true,
  citecolor=blue,
  linkcolor=blue,
  anchorcolor=blue
}
\usepackage{todonotes}
\usepackage{tikz-cd}
\usepackage{dsfont,mathrsfs}
\usepackage{graphicx}

\usepackage[noadjust,nocompress]{cite}

\ifdefined\newaliascnt
\else
  \usepackage{aliascnt}
\fi

\newtheorem{theorem}{Theorem}[section]
\newaliascnt{claim}{theorem}

\aliascntresetthe{claim}

\newaliascnt{lemma}{theorem}
\newtheorem{lemma}[lemma]{Lemma}
\aliascntresetthe{lemma}

\newaliascnt{proposition}{theorem}

\aliascntresetthe{proposition}

\newaliascnt{corollary}{theorem}
\newtheorem{corollary}[corollary]{Corollary}
\aliascntresetthe{corollary}

\newaliascnt{fact}{theorem}
\newtheorem{fact}[fact]{Fact}
\aliascntresetthe{fact}

\theoremstyle{definition}
\newaliascnt{remark}{theorem}

\aliascntresetthe{remark}

\newaliascnt{definition}{theorem}
\newtheorem{definition}[definition]{Definition}
\aliascntresetthe{definition}

\newaliascnt{question}{theorem}

\aliascntresetthe{question}

\newaliascnt{conjecture}{theorem}

\aliascntresetthe{conjecture}

\DeclareMathOperator{\crit}{crit}

\DeclareMathOperator{\rank}{rank}

\let\leq\leqslant

\newcommand{\ZF}{\mathsf{ZF}}
\newcommand{\ZFC}{\mathsf{ZFC}}

\newcommand{\HS}{\mathsf{HS}}

\newcommand{\VP}{\mathsf{VP}}
\newcommand{\SVC}{\mathsf{SVC}}

\newcommand{\cC}{\mathcal C}

\newcommand{\cL}{\mathcal L}
\newcommand{\cS}{\mathcal S}

\newcommand{\PP}{\mathbb P}

\newcommand{\forces}{\mathrel{\Vdash}}

\newcommand{\1}{\mathds 1}

\newcommand{\tup}[1]{\langle#1\rangle}

\date{October 17, 2026}
\title[Short proof that VP is symmetrically stable]{Short proof that Vop\v{e}nka's Principle\\ is symmetrically stable}
\author{Asaf Karagila}
\address[A.~Karagila]{School of Mathematics
  University of Leeds
  Leeds, LS2~9JT, UK}
\email{karagila@math.huji.ac.il}
\thanks{The research was supported by UKRI Future Leaders Fellowship [MR/Y034058/1] and by Engineering and Physical Sciences Research Council [UKRI5109].}
\thanks{No data are associated with this article. For the purpose of open access, the authors have applied a CC-BY licence to any Author Accepted Manuscript version arising from this submission.}
\begin{document}
\begin{abstract}
We show that Vop\v{e}nka's Principle is stable under symmetric extensions. Recently, de Groot and Wo{\l}oszyn \cite{VP} proved this fact. We found the proof of this theorem unnecessarily complicated. In this short note we prove the same statement. We direct the reader to \cite{KS:Towards} for an exposition on symmetric extensions.
\end{abstract}
\maketitle
\section{Basic definitions}
We write $\rank(x)$ to denote the least $\alpha$ for which $x\subseteq V_\alpha$. Similarly, if $\PP$ is a forcing notion and $\dot x$ is a $\PP$-name, $\rank_\PP(\dot x)$ is the least $\alpha$ such that $\dot x\subseteq\PP\times V_\alpha^\PP$.
\begin{definition}
  We say that $\kappa$ is \emph{$\Sigma_n$-correct}, denoted by $\kappa\in C(n)$, if $V_\kappa\prec_{\Sigma_n} V$. Namely, for every $\Sigma_n$ formula, $\varphi(x)$, and $a\in V_\kappa$, $V\models\varphi(a)\iff V_\kappa\models\varphi(a)$
\end{definition}
\begin{lemma}\label{lemma:name-rank}
  Let $\alpha\in C(1)$ and a forcing notion $\PP\in V_\alpha$. If $\dot a\in V^\PP_\alpha$, then $\dot a\in V_\alpha$.
\end{lemma}
\begin{proof}
  First, note that if $\alpha\in C(1)$, then $\alpha$ is a limit ordinal and whenever $\beta,\gamma<\alpha$, then $\beta+\gamma<\alpha$ as well. Suppose now that $\PP\in V_\alpha$. We prove by induction on the rank of $\dot a\in V^\PP_\alpha$, that if $\rank_\PP(\dot a)=\beta$, then $\dot a\subseteq V_{\max\{\rank(\PP),\beta\}+\omega}$. As $\alpha>\beta$ is an indecomposable ordinal, this would imply the wanted conclusion.

  Suppose this is true for all the names which appear in $\dot a$. Let $\tup{p,\dot x}\in\dot a$, then using the standard Kuratowski pairing, $\rank(\tup{p,\dot x})<\max(\rank(\PP),\rank(\dot x))+\omega$. Therefore, $\rank(\dot a)\leq\max\{\rank(\PP),\beta)\}+\omega<\alpha$ as wanted.
\end{proof}
\begin{definition}[Mohammd {\cite[Definition~2.1]{Marwan}}]
  We say that $\delta$ is an \emph{$n$-choiceless-extendible cardinal} if for every $\alpha\in C(n)\setminus\delta$ and every $\gamma<\delta$ there is some $\beta$ and an elementary embedding $j\colon V_\alpha\to V_\beta$ with $\gamma<\crit(j)<\alpha<j(\delta)$ and $j(\delta)\in C(n)$.
\end{definition}
\begin{definition}
  We say that \emph{Vop\v{e}nka's Principle} ($\VP$) holds if for every first-order language, $\cL$, if $\cC$ is a proper class of $\cL$-structures, then there is one which elementarily embeds into the other. We write $\VP(\Pi_n)$ to indicate that $\cC$ is definable by a $\Pi_n$ formula, and we use boldface to indicate that arbitrary parameters are used.
\end{definition}
\begin{fact}[Mohammd {\cite[Theorem~3.1]{Marwan}}]\label{fact:vp-equiv}
  For every $n$, $\VP(\boldsymbol{\Pi}_{n+1})$ is equivalent to the existence of a proper class of $n$-choiceless-extendible cardinals. In particular, $\VP$ is equivalent to the meta-theoretic schema that for every $n$ there exists a proper class of $n$-choiceless-extendible cardinals.
\end{fact}
Hayut and Karagila \cite{HK:Critical} provided a lifting theorem for elementary embedding for symmetric extension. In this paper we only need the very rudimentary application given by \cite[Theorem~4.9]{HK:Critical}, stating that if $\cS\in V_{\crit(j)}$, then $j$ lifts.
\section{Preservation}
In this section, fix a model of $\ZF$, $V$, and a symmetric system, $\cS$. We will denote by $W$ the resulting symmetric extension. As standard, $V_\alpha$ and $W_\alpha$ denote the $\alpha$th initial segment in the von Neumann hierarchy of these two models.
\begin{lemma}\label{lemma:emb-lift}
  If $\alpha$ is $\Sigma_1$-correct and $j\colon V_\alpha\to V_\beta$ is an elementary embedding such that $\cS\in V_{\crit(j)}$, then $j$ lifts to $j\colon W_\alpha\to W_\beta$.
\end{lemma}
\begin{proof}
  We use \cite[Theorem~4.9]{HK:Critical}, which states that any elementary embedding whose critical point lies above the von Neumann rank of the symmetric system lifts. Let us check that if $\alpha$ is indecomposable, then $j\colon V_\alpha\to V_\beta$ lifts to $j\colon W_\alpha\to W_\beta$. First, as $\alpha>\rank(\PP)$ is a limit ordinal, $W_\alpha$ is the symmetric extension of $V_\alpha$ by $\cS$, and similarly $W_\beta$ is the symmetric extension of $V_\beta$ by $\cS$. As $j(\cS)=\cS$, the embeddings lift correctly.
\end{proof}
\begin{lemma}\label{lemma:correctness-lift}
  Let $V$ be a model of $\ZF$ and $\lambda\in C(n)$. If $\cS\in V_\lambda$, then $\1\forces_\cS\check\lambda\in C(n)$.
\end{lemma}
\begin{proof}
  We prove this by induction on $n$. For $n=0$ this is trivial, as $W_\alpha$ is a transitive set. Let $a\in W_\alpha$ and let $\dot a\in\HS$ be a name for $a$. Then we have that $\dot a\in V_\alpha$. Suppose that $p\in\PP$ such that $p\forces_\cS\exists x\varphi(x,\dot a)$, where $\varphi$ is a $\Sigma_0$-formula.

  Then ``$p\forces_\cS\exists x\varphi(x,\dot a)$'' is itself a $\Sigma_1$-formula with $\PP$ as a parameter. It states that there exists densely many extensions $q\leq p$, and names $\dot x\in\HS$, such that $q\forces_\cS\varphi(\dot x,\dot a)$. As $\varphi$ is a bounded formula, this is the same as $q\forces\varphi(\dot x,\dot a)$. It is not hard to check that this is a $\Sigma_0$-statement in $V_\alpha$, using the appropriate parameters. As $\alpha\in C(1)$, $V$ and $V_\alpha$ agree, and the conclusion follows. The proof is similar for $\Pi_1$, when $W_\alpha\models\forall x\varphi(x,a)$, and the induction is standard.
\end{proof}
The following is a trivial application of \autoref{fact:vp-equiv} with the two previous lemmas.
\begin{theorem}
  If $V\models\VP$, then every symmetric extension satisfies $\VP$.\qed
\end{theorem}
As every generic extension is also a symmetric extension, then the following is a trivial corollary of the theorem, based on the equivalence of $\SVC$ with ``there is a partial order which forces the Axiom of Choice''.
\begin{corollary}
  If $V\models\VP+\SVC$, then $\ZFC+\VP$ holds in a generic extension.\qed
\end{corollary}
\providecommand{\bysame}{\leavevmode\hbox to3em{\hrulefill}\thinspace}
\providecommand{\MR}{\relax\ifhmode\unskip\space\fi MR }
\providecommand{\MRhref}[2]{%
  \href{http://www.ams.org/mathscinet-getitem?mr=#1}{#2}
}
\providecommand{\href}[2]{#2}

\end{document}